\documentclass[a4paper,12pt]{article}
\usepackage{color}
\usepackage[english]{babel}
\usepackage{amsfonts}
\usepackage{amsmath}
\usepackage{amsthm}
\usepackage{amssymb}
\usepackage{bbm}
\usepackage{mathrsfs}
\usepackage{esint}
\usepackage[utf8]{inputenc}
\usepackage{afterpage}
\usepackage{graphicx}

\newcommand{\N}{\mathbb{N}}

\newcommand{\R}{\mathbb{R}}

\renewcommand{\epsilon}{\varepsilon}
\renewcommand{\phi}{\varphi}
\newtheorem{defi}{Definition}[section]
\newtheorem{lemma}{Lemma}[section]
\newtheorem{thm}{Theorem}[section]
\newtheorem{prop}{Proposition}[section]
\newtheorem{rmk}{Remark}[section]
\newtheorem{ex}{Example}[section]
\numberwithin{equation}{section}
\DeclareMathOperator*{\esssup}{ess \, sup}

\begin{document}
\title{\textbf{On a singular Robin problem with convection terms}}
\author{
\bf Umberto Guarnotta, Salvatore A. Marano\thanks{Corresponding Author}\\
\small{Dipartimento di Matematica e Informatica, Universit\`a degli Studi di Catania,}\\
\small{Viale A. Doria 6, 95125 Catania, Italy}\\
\small{\it E-mail: umberto.guarnotta@gmail.com, marano@dmi.unict.it}\\
\mbox{}\\
\bf Dumitru Motreanu\\
\small{D\'{e}partement de Math\'{e}matiques, Universit\'{e} de Perpignan,}\\
\small{66860 Perpignan, France}\\
\small{\it E-mail: motreanu@univ-perp.fr}
}
\date{}
\maketitle
\begin{abstract}
In this paper, the existence of smooth positive solutions to a Robin boundary-value problem with non-homogeneous differential operator and reaction given by a nonlinear convection term plus a singular one is established. Proofs chiefly exploit sub-super-solution and truncation techniques, set-valued analysis, recursive methods, nonlinear regularity theory, as well as fixed point arguments. A uniqueness result is also presented. 
\end{abstract}
\vspace{2ex}
\vspace{2ex}

\section{Introduction}

Let $\Omega\subseteq\R^N$ ($N\geq 3$) be a bounded domain with a $ C^2$-boundary $\partial \Omega$ and let $f:\Omega\times \R\times\R^N\to [0,+\infty)$, $g:\Omega\times (0,+\infty)\to [0,+\infty)$ be two  Carathéodory functions. In this paper, we study existence and uniqueness of solutions to the following Robin problem:
\begin{equation}\label{problem} \tag{${\rm P}$}
\left\{
\begin{array}{ll}
- {\rm \, div} a(\nabla u)=f(x,u,\nabla u) + g(x,u)\;\; &\mbox{in}\;\;\Omega, \\
u > 0\;\; &\mbox{in}\;\;\Omega, \\
\displaystyle{\frac{\partial u}{\partial \nu_a}}+\beta |u|^{p-2}u= 0\;\; &\mbox{on}\;\;\partial \Omega,
\end{array}
\right.
\end{equation}
where $a:\R^N\to\R^N$ denotes a continuous strictly monotone map having suitable properties, which basically stem from Liebermann's nonlinear regularity theory \cite{Li} and Pucci-Serrin's maximum principle \cite{PS}; see Section \ref{S2} for details. Moreover, $\beta>0$, $1<p<+\infty$, while $\frac{\partial}{\partial \nu_a}$ denotes the co-normal derivative associated with $a$.

This problem gathers together several hopefully interesting technical features, namely:
\begin{itemize}
\item The involved differential operator appears in a general form that includes non-homogeneous cases.
\item $f$ depends on the solution and its gradient. So, the reaction exhibits nonlinear convection terms.  
\item $g$ can be singular at zero, i.e., $\displaystyle{\lim_{s\to 0^+}}g(x,s)=+\infty$.
\item Robin boundary conditions are imposed instead of (much more frequent) Dirichlet ones.
\end{itemize}
All these things have been extensively investigated, although separately. For instance, both differential operator and Robin conditions already appear in \cite{GMP} where, however, the problem has a fully variational structure, whilst \cite{PW} falls inside non-variational settings. The paper \cite{FarMotPug} addresses the presence of convection terms; see also \cite{MMM,MW,ZLM}, which exhibit more general contexts. Last but not least, singular problems were considered especially after the seminal works of Crandall-Rabinowitz-Tartar \cite{CRT} and Lazer-McKenna \cite{LM}. Among recent contributions on this subject, we mention \cite{GP2,PapWink}. Finally, \cite{LMZ} treats a $p$-Laplacian Dirichlet problem whose right-hand side has the same form as that in \eqref{problem}. It represented the starting point of our research.

Several issues arise when passing from Dirichlet to Robin boundary conditions. Accordingly, here, we try to develop some useful tools in this direction, including the localization of solutions to an auxiliary variational problem inside an opportune sublevel of its energy functional, constructed for preserving some compactness and semicontinuity properties (cf. Section \ref{S3}). 

Our main result, Theorem \ref{existence}, establishes the existence of a regular solution to \eqref{problem} chiefly via sub-super-solution and truncation techniques, set-valued analysis, recursive methods, nonlinear regularity theory, as well as Schaefer's fixed point theorem. Uniqueness is also addressed, but only when $ p=2$ (vide Section \ref{S4}). 

Usually, linear problems possess only one solution, whereas multiplicity is encountered in nonlinear phenomena. Hence, it might be of interest to seek hypotheses on $f$ and $g$ that yield uniqueness even if $p \neq 2$. As far as we know, this is still an open problem.
\section{Preliminaries}\label{S2}
Let $X$ be a set and let $C\subseteq X$. We denote by $\chi_C$ the characteristic function of $C$. If $C\neq\emptyset$ and
$\Gamma:C\to C$ then
\[
{\rm Fix}(\Gamma):=\{ x\in C: x=\Gamma(x)\}
\]
is the fixed point set of $\Gamma$. The following result, usually called Schaefer's theorem \cite[p. 827]{GP} or Leray-Schauder's alternative principle, will play a basic role in the sequel.
\begin{thm}\label{schaefer}	
Let $X$ be a Banach space, let $C \subseteq X$ be nonempty convex, and let $\Gamma: C\to C$ be continuous. Suppose $\Gamma$ maps bounded sets into relatively compact sets. Then either $\{ x\in C:x = t\,\Gamma(x)\;\mbox{for some}\; t\in (0,1)\} $ turns out unbounded or ${\rm Fix}(\Gamma)\neq\emptyset$.
\end{thm}
Given a partially ordered set $(X,\leq)$, we say that $X$ is downward directed when for every $x_1,x_2\in X$ there exists $x\in X$
such that $x\leq x_i$, $i=1,2$. The notion of upward directed set is analogous.

If $Y$ is a real function space on a set $\Omega\subseteq\R^N$ and $u,v\in Y$, then $u\leq v$ means $u(x)\leq v(x)$ for almost every $x\in\Omega$. Moreover, $Y_+ :=\{ u\in Y: u\geq 0\}$, $\Omega(u\leq v):=\{x\in\Omega: u(x)\leq v(x)\}$, etc. 

Let $X,Y$ be two metric spaces and let $\mathscr{S}:X\to 2^Y$. The multifunction $\mathscr{S}$ is called lower semicontinuous when for every $x_n\to x$ in $X$, $y\in\mathscr{S}(x)$ there exists a sequence $\{y_n\}\subseteq Y$ having the following properties: $y_n\to y$ in $Y$; $y_n\in\mathscr{S}(x_n)$ for all $n\in\N$. 

Finally, if $X$ is a Banach space and $J\in C^1(X)$, then
\[
{\rm Crit}(J):=\{ x \in X:J'(x) = 0\}
\]
is the critical set of $J$. 

The monograph \cite{CLM} represents a general reference on these topics.

Given any $s>1$, the symbol $s'$ will indicate the conjugate exponent of $s$, namely $s':=\frac{s}{s-1}$.

Henceforth, for $1<p<+\infty$, $\beta>0$, $\Omega$ as in the Introduction, and $u:\overline{\Omega} \to\R$ appropriate, the notation below will be adopted:
\[
\|u\|_\infty :=\esssup_{x \in \Omega} |u(x)|\, ;\quad \|u\|_{C^1(\overline{\Omega})}:=\|u\|_\infty + \|\nabla u\|_\infty\, ;
\]
\[
\|u\|_p:=\left( \int_{\Omega} |u|^p dx \right)^{\frac{1}{p}}\, ;\quad \|u\|_{p,\partial \Omega} := \left( \int_{\partial \Omega}
|u|^p d\sigma \right)^{\frac{1}{p}}\, ;
\]
\[
\|u\|_{1,p}:=\left( \|u\|_p^p+\|\nabla u\|_p^p\right)^\frac{1}{p}\, ;\quad \|u\|_{\beta,1,p}:=\left( \beta\|u\|_{p,\partial\Omega}^p+\|\nabla u\|_p^p \right)^\frac{1}{p}\, .
\]
Here, $\sigma$ denotes the $(N-1)$-dimensional Hausdorff measure on $\partial\Omega$. If $\nu (x)$ is the outward unit normal vector to $\partial\Omega$ at its point $x$ then $\frac{\partial}{\partial\nu_a}$ stands for the co-normal derivative associated with $a$, defined extending the map $u\mapsto\langle a(\nabla u),\nu\rangle$ from $C^1(\overline{\Omega})$ to $W^{1,p}(\Omega)$.
\begin{rmk}
 The trace inequality ensures that $\|u\|_{p,\partial \Omega}$ makes sense whenever $ u \in W^{1,p}(\Omega) $; see for instance \cite{E} or \cite{KJF}.
\end{rmk}
\begin{rmk}
It is known \cite{FMP} that
\[
{\rm int}(C^1(\overline{\Omega})_+)=\left\{u \in C^1(\overline{\Omega}): u(x) > 0\;\forall\, x\in\overline{\Omega} \right\}.
\]
\end{rmk}
\begin{rmk}
If $\beta>0$, then $\| \cdot \|_{\beta,1,p}$ is a norm on $W^{1,p}(\Omega)$ equivalent to $\|\cdot \|_{1,p}$. In particular, there exists $c_1=c_1(p,\beta,\Omega)\in (0,1) $ such that
\begin{equation}\label{equivnorm}
c_1\|u\|_{1,p}\leq\|u\|_{\beta,1,p}\leq\frac{1}{c_1}\|u\|_{1,p}\quad\forall\, u \in W^{1,p}(\Omega)\, .
\end{equation}
For the proof we refer to \cite{PW}.
\end{rmk}
Let $\omega\in C^1(0,+\infty)$ satisfy
\begin{equation*}
C_1 \leq\frac{t \omega'(t)}{\omega(t)} \leq C_2\, , \quad C_3 t^{p-1} \leq \omega(t) \leq C_4 (1+t^{p-1})
\end{equation*}
in $(0,+\infty)$, with $C_i$ suitable positive constants. We say that the operator $a:\R^N\to\R^N$ fulfills assumption 
$\underline{\rm{H(a)}}$ when:
\begin{itemize}
\item[$(a_1)$] $a(\xi) = a_0(|\xi|)\xi $ for all $\xi \in \R^N $, where $ a_0:(0,+\infty)\to(0,+\infty)$ is $C^1$, $t \mapsto ta_0(t)$ turns out strictly increasing, and
\begin{equation*}
\lim_{t \to 0^+} ta_0(t) = 0, \quad \lim_{t \to 0^+} \frac{ta_0'(t)}{a_0(t)} > -1.
\end{equation*}
\item[$(a_2)$] $ \displaystyle{|Da(\xi)| \leq C_5 \frac{\omega(|\xi|)}{|\xi|}} $ in $\R^N \setminus \{0\}$.
\item[$(a_3)$] $\displaystyle{\langle Da(\xi)y,y\rangle\geq\frac{\omega(|\xi|)}{|\xi|} |y|^2} $ for every $y,\xi\in\R^N$,
$\xi \neq 0$.
\end{itemize}
\begin{ex}
Various differential operators comply with ${\rm H(a)}$. Three classical examples are listed below.
\begin{itemize}
\item The so-called $p$-Laplacian: $\Delta_p u:={\rm div}\left( |\nabla u|^{p-2} \nabla u\right)$, which stems from $a_0(t):=t^{p-2}$.
\item The $(p,q)$-Laplacian: $\Delta_p u+\Delta_q u$, where $1< q< p< +\infty$. In this case, $a_0(t):=t^{p-2}+t^{q-2}$.
\item The generalized $p$-mean curvature operator:
\[
u \mapsto{\rm div}\left[ (1+|\nabla u|^2)^{\frac{p-2}{2}} \nabla u \right],
\]
corresponding to $a_0(t):=(1+t^2)^{\frac{p-2}{2}}$.
\end{itemize}
\end{ex}
Finally, define
\[
G_0(t):=\int_{0}^{t} s a_0(s){\rm d}s\;\;\forall\, t\in\R\quad\mbox{as well as}\quad G(\xi):=G_0(|\xi|)\;\;\forall\,\xi\in\R^N.
\]
\begin{prop}\label{opestimate}
Under hypothesis $ {\rm H(a)} $, there exists $ c_2 \in (0,1) $ such that
\[
|a(\xi)|\leq\frac{1}{c_2}(1+|\xi|^{p-1})\quad\mbox{and}\quad c_2 |\xi|^p\leq \langle a(\xi),\xi\rangle\leq\frac{1}{c_2}(1+|\xi|^p)
\]
for all $\xi\in\R^N$. In particular,
\[
c_2 |\xi|^p \leq G(\xi) \leq \frac{1}{c_2}(1+|\xi|^p)\,,\;\;\xi \in \R^N.
\]
\end{prop}
\begin{proof}
See \cite[Lemmas 2.1--2.2]{GMP} or \cite[Lemma 2.2 and Corollary 2.3]{PW}.
\end{proof}
\section{Existence}\label{S3}
Throughout this section, the convection term $f$ and the singularity $g$ will fulfill the assumptions below where, to avoid unnecessary technicalities, `for all $ x $' takes the place of `for almost all $ x $'. \\

\noindent $\underline{\rm{H(f)}}$
$f: \Omega \times \R \times\R^N \to [0,+\infty)$ is a Carathéodory function. Moreover, to every $M > 0$ there correspond  $c_M, d_M > 0$ such that
\[
f(x,s,\xi)\leq c_M + d_M |s|^{p-1}\quad \forall\, (x,s,\xi) \in\Omega\times\R\times\R^N\;\;\text{with}\;\; |\xi|\leq M.
\]
$ \underline{\rm{H(g)}}$
$g: \Omega \times (0,+\infty) \to [0,+\infty) $ is a Carathéodory function having the properties:
\begin{itemize}
\item[$\rm{(g_1)}$] $g(x,\cdot)$ turns out nonincreasing on $(0,1]$ whatever $x \in\Omega$, and $ g(\cdot,1)\not\equiv 0$.
\item[$\rm{(g_2)}$] There exist $c,d>0$ such that
\[
g(x,s)\leq c+d s^{p-1} \quad \forall\, (x,s) \in\Omega \times (1,+\infty).
\]
\item[$ \rm{(g_3)}$] With appropriate $\theta\in\rm{int}(C^1(\overline{\Omega})_+)$ and $\epsilon_0 > 0$, the map $x\mapsto g(x,\epsilon \theta(x)) $ belongs to $ L^{p'}(\Omega) $ for any $\epsilon \in (0,\epsilon_0)$. 
\end{itemize}
The paper \cite{LMZ} contains meaningful examples of functions $g$ that satisfy H(g).

Fix $w\in C^1(\overline{\Omega})$. We first focus on the singular problem (without convection terms)
\begin{equation}\label{auxprob} \tag{${\rm P}_w$} 
\left\{
\begin{array}{ll}
- {\rm div} \, a(\nabla u) = f(x,u,\nabla w) + g(x,u)\;\; &\mbox{in}\;\;\Omega, \\
u > 0 \;\; &\mbox{in}\;\; \Omega, \\
\displaystyle{\frac{\partial u}{\partial \nu_a}} + \beta |u|^{p-2}u= 0 \;\; &\mbox{on}\;\; \partial \Omega.
\end{array}
\right.
\end{equation}
\begin{defi}
$u\in W^{1,p}(\Omega)$ is called a subsolution to \eqref{auxprob} when
\begin{equation}\label{subsol}
\int_{\Omega} \langle a(\nabla u),\nabla v \rangle{\rm d}x
+ \beta\int_{\partial \Omega} |u|^{p-2}uv{\rm d}\sigma
\leq \int_{\Omega}[f(\cdot,u,\nabla w)+g(\cdot,u)]v{\rm d}x
\end{equation}
for all $v\in W^{1,p}(\Omega)_+ $. The set of subsolutions will be denoted by $ \underline{U}_w $. \\
We say that $u\in W^{1,p}(\Omega)$ is a supersolution to \eqref{auxprob} if
\begin{equation}\label{super}
\int_{\Omega} \langle a(\nabla u),\nabla v \rangle {\rm d}x + \beta\int_{\partial \Omega} |u|^{p-2}uv {\rm d}\sigma
\geq \int_{\Omega} [f(\cdot,u,\nabla w)+g(\cdot,u)]v {\rm d}x
\end{equation}
for every $v \in W^{1,p}(\Omega)_+ $, and indicate with $ \overline{U}_w $ the supersolution set. \\
Finally, $u\in W^{1,p}(\Omega)$ is called a solution of \eqref{auxprob} provided
\[
\int_{\Omega} \langle a(\nabla u),\nabla v \rangle {\rm d}x + \beta\int_{\partial \Omega} |u|^{p-2}uv {\rm d}\sigma = \int_{\Omega}[f(\cdot,u,\nabla w)+g(\cdot,u)]v {\rm d}x
\]
for all $ v \in W^{1,p}(\Omega)_+ $. The corresponding solution set will be denoted by $ U_w $. Obviously, $U_w =\overline{U}_w\cap \underline{U}_w $.
\end{defi}
\begin{lemma}\label{supersol}
If $ u_1,u_2 \in\overline{U}_w $ (resp. $u_1,u_2\in\underline{U}_w$), then $\min\{u_1,u_2\}\in\overline{U}_w$ (resp. $\max\{ u_1,u_2\} \in\underline{U}_w$). In particular, the set $\overline{U}_w $ (resp. $ \underline{U}_w $) is downward (resp. upward) directed.
\end{lemma}
\begin{proof} This proof is patterned after that of \cite[Lemma 10]{LMZ} (see also \cite{CLM}). Thus, we only sketch it. Pick  $u_1,u_2\in\overline{U}_w$, set $u:=\min\{u_1,u_2\} $, and define, for every $t\in\R$,
\[
\eta_\epsilon(t):= \left\{
\begin{array}{ll}
0 \quad &\mbox{when} \quad t < 0, \\
\frac{t}{\epsilon}\quad &\mbox{if} \quad 0 \leq t \leq \epsilon, \\
1 \quad &\mbox{for} \quad t > \epsilon,
\end{array}
\right.
\]
where $\epsilon>0$. Further, to shorten notation, write $\bar\eta_\epsilon(x):=\eta_\epsilon(u_2(x)-u_1(x))$. Evidently, both
$\bar\eta_\epsilon\in W^{1,p}(\Omega)_+$ and
$$\nabla\bar\eta_\epsilon=\eta'_\epsilon(u_2-u_1)\,\nabla(u_2-u_1).$$
Let $\hat v\in C^1(\overline{\Omega})_+$. Since $u_i$ fulfills \eqref{super}, one has
\begin{equation*}
\int_{\Omega}\langle a(\nabla u_i),\nabla v\rangle{\rm d}x + \beta \int_{\partial \Omega} |u_i|^{p-2}u_i v {\rm d}\sigma
\geq\int_{\Omega} [f(\cdot,u_i,\nabla w)+g(\cdot,u_i)]v {\rm d}x
\end{equation*}
whatever $v\in W^{1,p}(\Omega)_+ $. Choosing $v:=\bar\eta_\epsilon\,\hat v$ when $ i = 1 $, $v:= (1-\bar\eta_\epsilon)\hat v$ if $ i = 2 $, and adding term by term produces
\begin{equation}\label{termbyterm}
\begin{split}
&\int_{\Omega} \langle a(\nabla u_1) - a(\nabla u_2),\nabla (u_2-u_1) \rangle \eta_\epsilon'(u_2-u_1)\hat v {\rm d}x \\
&+ \int_{\Omega}\langle a(\nabla u_1), \nabla \hat v \rangle\,\bar\eta_\epsilon{\rm d}x
+ \int_{\Omega} \langle a(\nabla u_2),\nabla \hat v \rangle (1-\bar\eta_\epsilon) {\rm d}x \\
&+\beta\left( \int_{\partial \Omega} |u_1|^{p-2}u_1\bar\eta_\epsilon\hat v {\rm d}\sigma
+ \int_{\partial \Omega} |u_2|^{p-2}u_2 (1-\bar\eta_\epsilon)\hat v {\rm d}\sigma \right) \\
&\geq \int_{\Omega} [f(\cdot,u_1,\nabla w)+g(\cdot,u_1)]\bar\eta_\epsilon\hat v {\rm d}x\\
&+ \int_{\Omega} [f(\cdot,u_2,\nabla w)+g(\cdot,u_2)](1-\bar\eta_\epsilon)\hat v {\rm d}x.
\end{split}
\end{equation}
The strict monotonicity of $a$, combined with $\eta_\epsilon'(u_2-u_1)\hat v\geq 0$, lead to
\[
\int_{\Omega}\langle a(\nabla u_1)-a(\nabla u_2),\nabla (u_2-u_1)\rangle\eta_\epsilon'(u_2-u_1)\hat v{\rm d}x \leq 0.
\]
For almost every $x\in\Omega$ we have
\[
\nabla u(x) = \left\{
\begin{array}{ll}
\nabla u_1(x)\;\; &\mbox{if}\; u_1(x)<u_2(x), \\
\nabla u_2(x)\;\; &\mbox{otherwise,}
\end{array}
\right.
\]
as well as
\[
\lim_{\epsilon \to 0^+}\bar\eta_\epsilon(x)= \chi_{\Omega(u_1<u_2)}(x).
\]
Hence, letting $\epsilon\to 0^+$ and using the dominated convergence theorem, inequality \eqref{termbyterm} becomes
\begin{equation*}
\int_{\Omega}\langle a(\nabla u),\nabla\hat v\rangle {\rm d}x+\beta\int_{\partial \Omega}|u|^{p-2}u\hat v {\rm d}\sigma
\geq\int_{\Omega}[f(\cdot,u,\nabla w)+g(\cdot,u)]\hat v {\rm d}x;
\end{equation*}
see \cite[Lemma 10]{LMZ} for more details. Since $\hat v\in C^1(\overline{\Omega})_+$ was arbitrary, by density one arrives at
$u\in\overline{U}_w$.
\end{proof}
\begin{lemma}\label{subsollemma} 
Let ${\rm H(f)}$ and ${\rm H(g)}$ be satisfied. Then
there exists a subsolution $\underline{u}\in{\rm int}(C^1(\overline{\Omega})_+)$ to \eqref{auxprob} independent of $w$ and such that $\|\underline{u}\|_\infty \leq 1$.
\end{lemma}
\begin{proof}
Given any $\delta>0$, consider the problem
\begin{equation}\label{subprob}
\left\{
\begin{array}{ll}
- {\rm div} \, a(\nabla u) = \tilde{g}(x,u) \;\; &\mbox{in}\;\;\Omega, \\
\displaystyle{\frac{\partial u}{\partial \nu_a}} + \beta |u|^{p-2}u= 0\;\; &\mbox{on}\;\;\partial \Omega,
\end{array}
\right.
\end{equation}
where $\tilde{g}(x,s):=\min\{g(x,s),\delta\}$, $(x,s)\in\Omega\times (0,+\infty)$. Standard arguments yield a nontrivial solution $ \underline{u} \in W^{1,p}(\Omega) $ to \eqref{subprob}, because $\tilde{g}$ is bounded. Testing with $-\underline{u}^- $ we get
\[
-\int_{\Omega}\langle a(\nabla \underline{u}),\nabla\underline{u}^- \rangle {\rm d}x
-\beta \int_{\Omega}|\underline{u}|^{p-2} \underline{u} \underline{u}^- {\rm d}\sigma =
-\int_{\Omega} \tilde{g}(x,\underline{u}) \underline{u}^- {\rm d}x \leq 0,
\]
whence, by Proposition \ref{opestimate},
\[
c_2\|\underline{u}^-\|_{\beta,1,p}^p
\leq\int_{\Omega}\langle a(\nabla \underline{u}^-), \nabla \underline{u}^- \rangle {\rm d}x
+ \beta\int_{\Omega} (\underline{u}^-)^{p} {\rm d}\sigma \leq 0.
\]
Therefore, $\underline{u}\geq 0$. Regularity up to the boundary \cite{Li} and strong maximum principle \cite{PS} then force
$\underline{u}\in{\rm int}(C^1(\overline{\Omega})_+)$. Using the maximum principle one next has
\begin{equation}\label{subsupnorm}
\|\underline{u}\|_\infty \leq 1
\end{equation}
once $\delta$ is small enough. Let $\theta$ and $\epsilon_0$ be as in ${\rm (g_3)}$. Since $\underline{u},\theta\in{\rm int}(C^1(\overline{\Omega})_+) $, there exists $\epsilon\in (0,\epsilon_0)$ such that $\underline{u} - \epsilon \theta \in {\rm
int}(C^1(\overline{\Omega})_+)$. Via ${\rm (g_1)}$, \eqref{subsupnorm}, and ${\rm (g_3)} $, we thus infer
\begin{equation}\label{subsolsummability}
0 \leq g(\cdot,\underline{u}) \leq g(\cdot,\epsilon\theta)\in L^{p'}(\Omega).
\end{equation}
The conclusion is achieved by verifying that $\underline{u}\in\underline{U}_w$ for any $w\in C^1(\overline{\Omega})$. Pick such a $w$, test \eqref{subprob} with $v \in W^{1,p}(\Omega)_+ $, and recall the definition of $ \tilde{g} $, to arrive at
\begin{equation*}
\begin{split}
&\int_{\Omega} \langle a(\nabla \underline{u}),\nabla v\rangle{\rm d}x+\beta\int_{\partial \Omega}\underline{u}^{p-1}v {\rm d}\sigma
=\int_{\Omega}\tilde{g}(\cdot,\underline{u})v {\rm d}x \\
&\leq\int_{\Omega} g(\cdot,\underline{u})v{\rm d}x\leq\int_{\Omega} [f(\cdot,u,\nabla w) + g(\cdot,\underline{u})] v {\rm d}x,
\end{split}
\end{equation*}
as desired.
\end{proof}
\begin{rmk}
This proof shows that the subsolution $\underline{u}$ constructed in Lemma \ref{subsollemma} enjoys the further property:
\begin{equation}\label{subformula}
\int_{\Omega} \langle a(\nabla\underline{u}),\nabla v\rangle{\rm d}x
+\beta \int_{\partial \Omega}|\underline{u}|^{p-2}\underline{u}v {\rm d}\sigma
\leq\int_{\Omega} g(\cdot,\underline{u})v{\rm d}x\;\; \forall\, v \in W^{1,p}(\Omega)_+.
\end{equation}
\end{rmk}
\noindent Given $w\in C^1(\overline{\Omega})$, consider the truncated problem
\begin{equation}\label{truncprob} 
\left\{
\begin{array}{ll}
- {\rm div} \, a(\nabla u) =\hat{f}(x,u)+\hat{g}(x,u)\;\; &\mbox{in}\;\;\Omega, \\
u > 0\;\; &\mbox{in}\;\; \Omega, \\
\displaystyle{\frac{\partial u}{\partial \nu_a}} + \beta |u|^{p-2}u= 0\;\;&\mbox{on}\;\;\partial \Omega,
\end{array}
\right.
\end{equation}
where
\begin{equation}\label{1f}
\hat{f}(x,s):= \left\{
\begin{array}{ll}
f(x,\underline{u}(x),\nabla w(x))\;\; &\mbox{if}\;\;s \leq \underline{u}(x), \\
f(x,s,\nabla w(x))\;\; &\mbox{otherwise,}
\end{array}
\right.
\end{equation}
\begin{equation}\label{1}
\hat{g}(x,s):= \left\{
\begin{array}{ll}
g(x,\underline{u}(x))\;\; &\mbox{if}\;\; s \leq \underline{u}(x), \\
g(x,s)\;\; &\mbox{otherwise.}
\end{array}
\right.
\end{equation}
The energy functional corresponding to \eqref{truncprob} writes
\begin{equation*}
\begin{split}
\mathscr{E}_w(u):=\frac{1}{p} \int_{\Omega} G(\nabla u){\rm d}x+\frac{\beta}{p} \int_{\partial \Omega} |u|^p {\rm d}\sigma 
- \int_{\Omega}\hat{F}(\cdot,u){\rm d}x-\int_{\Omega} \hat{G}(\cdot,u){\rm d}x
\end{split}
\end{equation*}
for all $ u \in W^{1,p}(\Omega) $, with
\[
\hat{F}(x,s):=\int_{0}^{s} \hat{f}(x,t) {\rm d}t, \quad \hat{G}(x,s):=\int_{0}^{s} \hat{g}(x,t) {\rm d}t.
\]
Hypotheses ${\rm H(f)}$--${\rm H(g)}$ ensure that $\mathscr{E}_w$ is of class $ C^1 $ and weakly sequentially lower semicontinuous; see, e.g., \cite[Lemma 3.1]{GMP}. Under the additional condition
\begin{equation}\label{condition} 
d_M+d< c_1^p c_2 \quad \forall\, M >0,
\end{equation}
it  turns out also coercive, as the next lemma shows.
\begin{lemma}\label{estlemma} 
Let $\mathscr{B}$ be a nonempty bounded set in $C^1(\overline{\Omega})$. If ${\rm H(f)}$, ${\rm H(g)}$, and \eqref{condition} hold true then there exist $\alpha_1 \in (0,1)$, $\alpha_2 > 0$ such that
\begin{equation*}
\mathscr{E}_w(u) \geq\frac{\alpha_1}{p}\|u\|_{1,p}^p - \alpha_2(1+\|u\|_{1,p})\quad\forall\, (u,w)\in W^{1,p}(\Omega)\times\mathscr{B}.
\end{equation*}
\end{lemma}
\begin{proof}
Put $\hat{M}:=\displaystyle{\sup_{w \in \mathscr{B}}}\|w\|_{C^1(\overline{\Omega})}$. By \eqref{1f}--\eqref{1}, Proposition \ref{opestimate} entails
\begin{equation*}
\begin{split}
\mathscr{E}_w(u) &\geq\frac{c_2}{p} \|\nabla u\|_p^p+\frac{\beta}{p}\|u\|_{p,\partial \Omega}^p
-\int_{\Omega} [f(\cdot,\underline{u},\nabla w)+ g(\cdot,\underline{u})]\underline{u} {\rm d}x \\
&- \int_{\Omega(u>\underline{u})}\left(\int_{\underline{u}}^{u} f(\cdot,t,\nabla w) {\rm d}t \right){\rm d}x
-\int_{\Omega(u>\underline{u})} \left(\int_{\underline{u}}^{u} g(\cdot,t) {\rm d}t \right) {\rm d}x.
\end{split}
\end{equation*}
Hypothesis ${\rm H(f)} $ along with H\"older's inequality imply
\begin{equation*}
\begin{split}
\int_{\Omega(u>\underline{u})}\left( \int_{\underline{u}}^{u} f(\cdot,t,\nabla w){\rm d}t \right){\rm d}x
&\leq\int_{\Omega(u>\underline{u})}\left( \int_{0}^{u} f(\cdot,t,\nabla w){\rm d}t \right){\rm d}x \\
&\leq c_{\hat{M}}|\Omega|^{\frac{1}{p'}} \|u\|_p + \frac{d_{\hat{M}}}{p} \|u\|_p^p \\
&\leq c_{\hat{M}} |\Omega|^{\frac{1}{p'}} \|u\|_{1,p} +
\frac{d_{\hat{M}}}{p} \|u\|_{1,p}^p.
\end{split}
\end{equation*}
Exploiting \eqref{subsupnorm}, ${\rm (g_2)}$, and H\"older's inequality again, we have
\begin{equation*}
\begin{split}
&\int_{\Omega(u>\underline{u})} \left(\int_{\underline{u}}^{u} g(\cdot,t) {\rm d}t \right) {\rm d}x \\
&\leq\int_{\Omega(u>\underline{u})}\left( \int_{\underline{u}}^{1} g(\cdot,t) {\rm d}t \right) {\rm d}x
+ \int_{\Omega(u>1)}\left( \int_{1}^{u} g(\cdot,t) {\rm d}t \right) {\rm d}x \\
&\leq \int_{\Omega(u>\underline{u})} g(\cdot,\underline{u}) {\rm d}x
+\int_{\Omega(u>1)}\left( \int_1^{u} (c+dt^{p-1}) {\rm d}t \right) {\rm d}x \\
&\leq \int_{\Omega} g(\cdot,\underline{u}) {\rm d}x+ c |\Omega|^{\frac{1}{p'}}\|u\|_p + \frac{d}{p} \|u\|_p^p \\
& \leq\int_{\Omega} g(\cdot,\underline{u}) {\rm d}x+ c |\Omega|^{\frac{1}{p'}} \|u\|_{1,p} + \frac{d}{p} \|u\|_{1,p}^p.
\end{split}
\end{equation*}
Hence, through \eqref{equivnorm} we easily arrive at
\begin{equation*}
\begin{split}
\mathscr{E}_w(u) &\geq\frac{c_2}{p}\|u\|_{\beta,1,p}^p-\frac{d_{\hat{M}}+d}{p} \|u\|_{1,p}^p
- (c_{\hat{M}} + c) |\Omega|^{\frac{1}{p'}}\|u\|_p - K \\
&\geq\frac{c_1^p c_2 - d_{\hat{M}} - d}{p} \|u\|_{1,p}^p- (c_{\hat{M}} + c) |\Omega|^{\frac{1}{p'}} \|u\|_{1,p} - K\\
&\geq\frac{c_1^p c_2 - d_{\hat{M}} - d}{p}\|u\|_{1,p}^p-\max\{(c_{\hat{M}}+c) |\Omega|^{\frac{1}{p'}},K\} (1+\|u\|_{1,p}),
\end{split}
\end{equation*}
where
\begin{equation*}
\begin{split}
K &:=\int_{\Omega} [f(\cdot,\underline{u},\nabla w)]+g(\cdot,\underline{u})]\underline{u} {\rm d}x
+\int_{\Omega} g(\cdot,\underline{u}) {\rm d}x \\
&\leq\int_{\Omega} (c_{\hat{M}} + d_{\hat{M}}){\rm d}x + 2 \int_{\Omega} g(\cdot,\epsilon \theta) {\rm d}x 
\leq (c_{\hat{M}} + d_{\hat{M}}) |\Omega| + 2 \|g(\cdot, \epsilon
\theta)\|_{p'} |\Omega|^{\frac{1}{p}}
\end{split}
\end{equation*}
due to ${\rm H(f)}$ and \eqref{subsupnorm}--\eqref{subsolsummability}. Now, the conclusion follows from \eqref{condition}. 
\end{proof}
\begin{rmk}\label{regularity}
A standard application of Moser's iteration technique \cite{Le} shows that any solution to \eqref{truncprob} lies in
$L^\infty(\Omega)$. By Liebermann's regularity theory \cite{Li}, it actually is H\"older continuous up to the boundary.
\end{rmk}
\begin{lemma}\label{criticalprop} 
Let  ${\rm H(f)}$, ${\rm H(g)}$, and \eqref{condition} be satisfied. Then
\[
\emptyset \neq {\rm Crit}(\mathscr{E}_w) \subseteq U_w \cap \{ u \in C^1(\overline{\Omega}): u \geq \underline{u}\}.
\]
\end{lemma}
\begin{proof}
Since $\mathscr{E}_w$ is coercive (cf. Lemma \ref{estlemma}), the Weierstrass-Tonelli theorem produces ${\rm Crit}(\mathscr{E}_w)\neq\emptyset$. Pick any $ u \in {\rm Crit}(\mathscr{E}_w) $,  test \eqref{truncprob} with $ (\underline{u}-u)^+ $, and exploit \eqref{1f}--\eqref{1}, besides \eqref{subformula}, to achieve
\begin{equation*}
\begin{split}
&\int_{\Omega} \langle a(\nabla u),\nabla(\underline{u}-u)^+ \rangle {\rm d}x + \beta \int_{\partial \Omega} |u|^{p-2}u(\underline{u}-u)^+ {\rm d}\sigma \\
&= \int_{\Omega} [\hat{f}(\cdot,u) + \hat{g}(\cdot,u)](\underline{u}-u)^+ {\rm d}x \\
&\geq\int_{\Omega}\hat{g}(\cdot,u)(\underline{u}-u)^+{\rm d}x
=\int_{\Omega} g(\cdot,\underline{u})(\underline{u}-u)^+{\rm d}x \\
&\geq \int_{\Omega} \langle a(\nabla\underline{u}),\nabla(\underline{u}-u)^+ \rangle {\rm d}x
+ \beta\int_{\partial \Omega}|\underline{u}|^{p-2}\underline{u}(\underline{u}-u)^+ {\rm d}\sigma.
\end{split}
\end{equation*}
Rearranging terms we get
\[
\int_{\Omega} \langle a(\nabla \underline{u}) - a(\nabla
u),\nabla(\underline{u}-u)^+ \rangle {\rm d}x + \beta \int_{\partial
\Omega} (|\underline{u}|^{p-2}\underline{u} -
|u|^{p-2}u)(\underline{u}-u)^+ {\rm d}\sigma \leq 0.
\]
The strict monotonicity of $a$, combined with \cite[Lemma A.0.5]{P}, entail
\[
\nabla(\underline{u}-u)^+ =0\;\;\text{in}\;\; \Omega,\quad(\underline{u}-u)^+ =0\;\;\text{on}\;\; \partial \Omega.
\]
So, $\|(\underline{u}-u)^+\|_{\beta,1,p} = 0$, which means $u\geq \underline{u}$. Finally, by \eqref{1f}--\eqref{1} one has $u \in U_w $, while $u\in C^1(\overline{\Omega})$ according to Remark \ref{regularity}.
\end{proof}
For every $w\in C^1(\overline{\Omega})$ we define
\[
\mathscr{S}(w):=\{ u\in C^1(\overline{\Omega}): u \in U_w,\, u \geq \underline{u},\, \mathscr{E}_w(u) < 1 \}.
\]
\begin{lemma}\label{Scompact}
Under assumptions ${\rm H(f)}$, ${\rm H(g)}$, and \eqref{condition}, the multifunction $\mathscr{S}:C^1(\overline{\Omega})\to 2^{C^1(\overline{\Omega})}$ takes nonempty values and maps bounded sets into relatively compact sets.
\end{lemma}
\begin{proof}
If $ w\in C^1(\overline{\Omega})$, then there exists $ \hat u_w\in {\rm Crit}(\mathscr{E}_w) $ such that
\[
\hat u_w\in C^1(\overline{\Omega}),\quad\hat u_w\geq\underline{u},\quad\mathscr{E}_w(\hat u_w) = \inf_{W^{1,p}(\Omega)} \mathscr{E}_w \leq
\mathscr{E}_w(0) = 0 < 1;
\]
cf. the proof of Lemma \ref{criticalprop}. Hence, $\mathscr{S}(w)\neq\emptyset$, because $\hat u_w \in \mathscr{S}(w)$. Let
$\mathscr{B}\subseteq C^1(\overline{\Omega}) $ nonempty bounded. From Lemma \ref{estlemma} it follows
\[
\frac{\alpha_1}{p}\|u\|_{1,p}^p-\alpha_2(1+\|u\|_{1,p})\leq\mathscr{E}_w(u)<1\;\;\forall\, u\in\mathscr{S}(w),\;w\in \mathscr{B}, 
\]
whence $\mathscr{S}(\mathscr{B})$ turns out bounded in $W^{1,p}(\Omega)$. By nonlinear regularity theory \cite{Li}, the same holds when $C^{1,\alpha}(\overline{\Omega})$, with suitable $ \alpha\in (0,1)$, replaces $W^{1,p}(\Omega)$. Recalling that $C^{1,\alpha}(\overline{\Omega})\hookrightarrow C^1(\overline{\Omega})$ compactly yields the conclusion.
\end{proof}

To see that $\mathscr{S}$ is lower semicontinuous, we shall employ the next technical lemma.

\begin{lemma}\label{reclemma} 
Let $\alpha,\beta,\gamma > 0$, let $1< p <+\infty$, and let $\{a_k\}\subseteq [0,+\infty)$ satisfy the recursive relation
\begin{equation}\label{reclaw} 
\alpha a_k^p\leq\beta a_k + \gamma a_{k-1}^p\;\;\forall\, k \in \N.
\end{equation}
If $\gamma<\alpha$, then the sequence $\{a_k\}$ is bounded.
\end{lemma}
\begin{proof}
Using the obvious inequality
\begin{equation*}
a_k \leq T + T^{1-p} a_k^p,\quad T > 0,
\end{equation*}
\eqref{reclaw} becomes
\begin{equation*}
\left(\alpha-\beta T^{1-p}\right) a_k^p\leq\beta T+\gamma a_{k-1}^p\;\;\forall\, k \in \N.
\end{equation*}
Since $\sigma:= 1/p < 1$, this entails
\begin{equation*}
\left(\alpha-\beta T^{1-p}\right)^\sigma a_k \leq\left(\beta T+\gamma a_{k-1}^p\right)^\sigma\leq (\beta T)^\sigma +
\gamma^\sigma a_{k-1}
\end{equation*}
or, equivalently,
\begin{equation}\label{recfinal} 
a_k \leq\left(\frac{\beta T}{\alpha-\beta T^{1-p}}\right)^\sigma
+\left(\frac{\gamma}{\alpha-\beta T^{1-p}}\right)^\sigma a_{k-1},\quad k \in \N,
\end{equation}
provided $T>0 $ is large enough. Choosing $T>\left( \frac{\beta}{\alpha - \gamma} \right)^{\frac{1}{p-1}}$, the coefficient of
$ a_{k-1} $ turns out strictly less than 1. A standard computation based on \eqref{recfinal} completes  the proof.
\end{proof}
\begin{lemma}\label{Slsc}
Suppose ${\rm H(f)}$--${\rm H(g)}$ hold and, moreover,
\begin{equation}\label{reccond}
d_M+d< \frac{c_1^p c_2}{p}\quad\forall\, M>0.
\end{equation}
Then the multifunction $\mathscr{S}:C^1(\overline{\Omega})\to 2^{C^1(\overline{\Omega})}$ is lower semicontinuous.
\end{lemma}
\begin{proof}
The proof is patterned after that of \cite[Lemma 20]{LMZ}. So, some details will be omitted. Let
\begin{equation}\label{convwn}
w_n \to w\;\;\mbox{in}\;\; C^1(\overline{\Omega}).
\end{equation}
We claim that to each $\tilde u\in\mathscr{S}(w)$ there corresponds a sequence $\{ u_n\}\subseteq C^1(\overline{\Omega})$ enjoying the following properties:
\[
u_n \in \mathscr{S}(w_n),\;\; n\in\N;\quad u_n\to\tilde u\;\;\text{in}\;\; C^1(\overline{\Omega}).
\]
Fix $\tilde u\in\mathscr{S}(w)$. For every $n\in\N$, consider the auxiliary problem
\begin{equation}\label{vwnprob} \tag{${\rm P}_{\tilde u,w_n}$}
\left\{
\begin{array}{ll}
- {\rm div}\, a(\nabla u) = f(x,\tilde u,\nabla w_n) + \hat{g}(x,\tilde u)\;\;&\mbox{in}\;\;\Omega, \\
u > 0\;\; &\mbox{in}\;\;\Omega, \\
\displaystyle{\frac{\partial u}{\partial \nu_a}}+\beta |u|^{p-2}u= 0\;\;&\mbox{on}\;\;\partial \Omega,
\end{array}
\right.
\end{equation}
with $\hat{g}(x,s)$ given by \eqref{1}. One has $\hat{g}(x,\tilde u)=g(x,\tilde u)$, because $\tilde u\in\mathscr{S}(w)$, while the associated energy functional writes
\begin{equation*}
\begin{split}
\mathscr{E}_{\tilde u,w_n}(u)&:=\frac{1}{p} \int_{\Omega} G(\nabla u){\rm d}x+\beta \int_{\partial \Omega}|u|^p {\rm d}\sigma\\
&-\int_{\Omega} f(x,\tilde u,\nabla w_n)u{\rm d}x-\int_{\Omega}\hat{g}(x,\tilde u)u{\rm d}x,\;\; u\in W^{1,p}(\Omega).
\end{split}
\end{equation*}
Since $\mathscr{E}_{\tilde u,w_n}$ turns out strictly convex, the same argument exploited to show Lemma \ref{criticalprop} yields here a unique solution $u_n^0\in{\rm int(C^1(\overline{\Omega})_+)}$ of \eqref{vwnprob} such that
\begin{equation}\label{convEn}
\mathscr{E}_{\tilde u,w_n}(u_n^0)\leq 0.
\end{equation}
Via \eqref{convwn}--\eqref{convEn}, reasoning as in Lemmas \ref{estlemma} and \ref{Scompact} (but for $\mathscr{E}_{\tilde u,w} $ instead of $ \mathscr{E}_w $ and $\mathscr{B}:=\{w_n:n \in \N\}$), we deduce that $\{u_n^0\}\subseteq C^1 (\overline{\Omega})$ is relatively compact. Consequently, $u_n^0\to u^0$ in $C^1(\overline{\Omega})$, where a subsequence is considered when necessary. By \eqref{convwn} again and Lebesgue's dominated convergence theorem, $u^0 $ solves problem
$({\rm P}_{\tilde u,w})$. Thus, a fortiori, $u^0 =\tilde u$, because $({\rm P}_{\tilde u,w})$ possesses one solution at most. An induction procedure provides now a sequence $\{u_n^k\}$ such that $u_n^k$ solves problem $({\rm P}_{u_n^{k-1},w_n})$, the inequality $\mathscr{E}_{u_n^{k-1},w_n}(u_n^k)\leq 0$ holds, and
\begin{equation}\label{nindex}
\lim_{n\to+\infty }u_n^k=\tilde u\;\;\mbox{in}\;\; C^1(\overline{\Omega})\;\;\mbox{for all}\;\; k\in\N.
\end{equation}
\underline{\rm Claim}: $\{u_n^k\}_{k \in \N}\subseteq C^1(\overline{\Omega})$ is relatively compact.\\ 
In fact, recalling \eqref{convwn}, pick $M=\displaystyle{\sup_{n\in\N}} \|w_n\|_{C^1(\overline{\Omega})}$. Through H\"older's and Young's inequalities, besides \eqref{subsolsummability}, we obtain
\begin{equation}\label{recprinc} 
\frac{1}{p} \int_{\Omega} G(\nabla u_n^k) {\rm d}x
+\frac{\beta}{p}\int_{\partial \Omega} |u_n^k|^p {\rm d}\sigma\geq\frac{c_1^p c_2}{p}\|u_n^k\|_{1,p}^p,
\end{equation}
\begin{equation}\label{recf}
\begin{split}
&\int_\Omega f(\cdot,u_n^{k-1},\nabla w_n) u_n^k {\rm d}x
\leq c_M |\Omega|^{\frac{1}{p'}}\|u_n^k\|_p + d_M \int_\Omega |u_n^{k-1}|^{p-1}|u_n^k| {\rm d}x \\
&\leq c_M |\Omega|^{\frac{1}{p'}} \|u_n^k\|_p + d_M \left(\frac{1}{p'} \|u_n^{k-1}\|_p^p + \frac{1}{p} \|u_n^k\|_p^p \right),
\end{split}
\end{equation}
as well as
\begin{equation}\label{recg}
\begin{split}
\int_\Omega & \hat{g}(\cdot,u_n^{k-1}) u_n^k {\rm d}x\\
&=\int_{\Omega(u_n^{k-1}\leq 1)} \hat{g}(\cdot,u_n^{k-1}) u_n^k {\rm d}x
+\int_{\Omega(u_n^{k-1}> 1)} \hat{g}(\cdot,u_n^{k-1}) u_n^k {\rm d}x \\
&\leq\int_{\Omega(u_n^{k-1}\leq 1)} g(\cdot,\underline{u}) u_n^k {\rm d}x+\int_{\Omega(u_n^{k-1}> 1)} g(\cdot,u_n^{k-1}) u_n^k {\rm d}x \\
&\leq(\| g(\cdot,\underline{u})\|_{p'}+c|\Omega|^{\frac{1}{p'}}) \|u_n^k\|_p
+d\int_\Omega |u_n^{k-1}|^{p-1}|u_n^k| {\rm d}x \\
&\leq(\| g(\cdot,\underline{u})\|_{p'}+c|\Omega|^{\frac{1}{p'}})\|u_n^k\|_p
+d\left(\frac{1}{p'} \|u_n^{k-1}\|_p^p + \frac{1}{p} \|u_n^k\|_p^p \right).
\end{split}
\end{equation}
Since $\mathscr{E}_{u_n^{k-1},w_n}(u_n^k) \leq 0$, estimates \eqref{recprinc}--\eqref{recg} entail
\begin{equation*}\label{recgeneral}
\begin{split}
&\frac{c_1^p c_2 - d_M - d}{p} \|u_n^k\|_{1,p}^p\\ 
&\leq\left(\| g(\cdot,\underline{u})\|_{p'}+(c_M + c)|\Omega|^{\frac{1}{p'}}\right)\|u_n^k\|_{1,p}
+\frac{d_M + d}{p'} \|u_n^{k-1}\|_{1,p}^p
\end{split}
\end{equation*}
for all $k\in\N$. Thanks to \eqref{reccond}, Lemma \ref{reclemma} applies, and the sequence $\{u_n^k\}_{k \in \N}$ turns out bounded in $W^{1,p}(\Omega)$. Standard arguments involving regularity up to the boundary (cf. the proof of Lemma \ref{Scompact}) yield the claim.

We may thus assume there exists $\{u_n\}\subseteq C^1(\overline{\Omega})$ fulfilling
\begin{equation}\label{kindex}
\lim_{k \to \infty} u_n^k = u_n\;\,\mbox{in}\;\; C^1(\overline{\Omega})
\end{equation}
whatever $n\in\N$. By \eqref{kindex} and Lebesgue's dominated convergence theorem one has $u_n\in U_{w_n}$. Moreover, as in the proof of Lemma \ref{criticalprop}, $u_n \geq \underline{u}$. Due to \eqref{nindex} and \eqref{kindex}, the double limit lemma \cite[Proposition A.2.35]{GP} gives
\begin{equation}\label{doublelimit}
u_n\to\tilde u\;\;\mbox{in}\;\; C^1(\overline{\Omega}).
\end{equation}
Thus, it remains to show that $\mathscr{E}_{w_n}(u_n) < 1$. From \eqref{convwn} we easily infer $\mathscr{E}_{w_n}(\tilde u)\to\mathscr{E}_{w}(\tilde u)$. Since $\mathscr{E}_{w_n}$ is of class $C^1$, via \eqref{convwn} and \eqref{doublelimit} one arrives at 
\[
\lim_{n\to+\infty}\left(\mathscr{E}_{w_n}(u_n) - \mathscr{E}_{w}(\tilde u)\right)=0,
\]
namely $\mathscr{E}_{w_n}(u_n)\to\mathscr{E}_{w}(\tilde u)$. This completes the proof, because $\tilde u\in\mathscr{S}(w)$,
whence $\mathscr{E}_{w}(\tilde u)<1$.
\end{proof}
\begin{lemma}\label{comparison} 
Under ${\rm H(f)}$, ${\rm H(g)}$, and \eqref{condition}, the set $\mathscr{S}(w)$, $w\in C^1(\overline{\Omega})$, is downward directed.
\end{lemma}
\begin{proof}
Let $u_1,u_2\in \mathscr{S}(w)$ and let $\hat{u}:=\min\{u_1,u_2\}$.  By Lemma \ref{supersol} we have $\hat u\in
\overline{U}_w$. Consider the problem
\begin{equation}\label{compareprob}
\left\{
\begin{array}{ll}
- {\rm div} \, a(\nabla u) = h(x,u)\;\; & \mbox{in}\;\;\Omega, \\
u > 0\;\; & \mbox{in}\;\;\Omega, \\
\displaystyle{\frac{\partial u}{\partial \nu_a}}+\beta |u|^{p-2}u= 0\;\; & \mbox{on}\;\;\partial \Omega,
\end{array}
\right.
\end{equation}
where
\begin{equation*}
h(x,s) = \left\{
\begin{array}{ll}
f(x,\underline{u}(x),\nabla w(x)) + g(x,\underline{u}(x))\;\; & \mbox{for}\; s \leq \underline{u}(x), \\
f(x,s,\nabla w(x)) + g(x,s)\;\;  & \mbox{if}\; \underline{u}(x) < s < \hat{u}(x), \\
f(x,\hat u(x),\nabla w(x)) + g(x,\hat u(x))\;\; & \mbox{when} \; s \geq\hat{u}(x).
\end{array}
\right.
\end{equation*}
The associated energy functional writes
\[
\tilde{\mathscr{E}}_w(u):=\frac{1}{p}\int_{\Omega} G(\nabla u){\rm d}x+\beta\int_{\partial \Omega} |u|^p {\rm d}x - \int_{\Omega}{\rm d}x \int_0^{u} h(\cdot,t){\rm d}t,\; u \in W^{1,p}(\Omega).
\]
Arguing as in Lemma \ref{Scompact} produces a solution $\tilde{u}\in C^1(\overline{\Omega})$ to \eqref{compareprob} such that
$\tilde{\mathscr{E}}_w(\tilde{u})\leq 0$. Next, adapt the proof of Lemma \ref{criticalprop} and exploit the fact that $\hat{u}$ is
a supersolution of \eqref{compareprob} to achieve $\underline{u}\leq\tilde{u}\leq\hat{u}$. Consequently, $\tilde{u}\in U_w$ and
\[
\mathscr{E}_w(\tilde{u})=\tilde{\mathscr{E}}_w(\tilde{u})\leq 0<1.
\]
This forces $\tilde{u}\in \mathscr{S}(w)$, besides $\tilde{u}\leq\min\{u_1,u_2\}$.
\end{proof}
\begin{lemma}\label{wellposed} 
If ${\rm H(f)}$, ${\rm H(g)}$, and \eqref{condition} hold true then for every $w\in C^1(\overline{\Omega})$ the set $\mathscr{S}(w)$ possesses absolute minimum.
\end{lemma}
\begin{proof}
Fix  $w \in C^1(\overline{\Omega})$. We already know (see Lemma \ref{comparison}) that $\mathscr{S}(w)$ turns out downward directed. If $\mathscr{C} \subseteq \mathscr{S}(w)$ is a chain in $\mathscr{S}(w)$ then there exists a sequence $\{u_n\}
\subseteq\mathscr{S}(w) $ satisfying
\[
\lim_{n \to \infty} u_n = \inf \mathscr{C}.
\]
On account of Lemma \ref{Scompact} and up to subsequences, one has $u_n\to\hat u$ in $C^1(\overline{\Omega})$. Thus, $\hat u =\inf\mathscr{C}$. By Zorn's Lemma, $\mathscr{S}(w)$ admits a minimal element $u_w$. It remains to show that $u_w=\min \mathscr{S}(w)$. Pick any $u\in\mathscr{S}(w)$. Through Lemma \ref{comparison} we get $\tilde{u}\in\mathscr{S}(w)$ such that
$\tilde{u} \leq\min\{u_w,u\}$. The minimality of $u_w$ entails $u_w=\tilde{u}$. Therefore, $u_w\leq u$, as desired.
\end{proof}
\begin{rmk}
This proof is patterned after the one in \cite[Theorem 23]{LMZ}. 
\end{rmk}
Lemma \ref{wellposed} allows to consider the function $\Gamma: C^1(\overline{\Omega})\to C^1(\overline{\Omega})$ given by
\[
\Gamma(w):=\min\mathscr{S}(w)\quad\forall\, w\in C^1(\overline{\Omega}).
\]
\begin{lemma}\label{propgamma}
Under assumptions ${\rm H(f)}$, ${\rm H(g)}$, and \eqref{reccond}, $\Gamma$  is continuous and maps bounded sets into relatively compact sets.
\end{lemma}
\begin{proof}
It is analogous to that of \cite[Lemma 24]{LMZ}. So, we will omit details. Let $\mathscr{B}\subseteq C^1(\overline{\Omega})$ be
bounded. Since $\Gamma(\mathscr{B})\subseteq\mathscr{S}(\mathscr{B})$ and $\mathscr{S}(\mathscr{B})$ turns out relatively compact (cf. Lemma \ref{Scompact}),  $\Gamma(\mathscr{B})$ enjoys the same property. Next, suppose $ w_n\to w$ in
$C^1(\overline{\Omega})$. Setting $u_n :=\Gamma (w_n)$, one evidently has $u_n \to u$ in $C^1(\overline{\Omega})$, where a subsequence is considered when necessary. The function $u$ complies with $u\geq\underline{u} $ and $\mathscr{E}_w(u) < 1$ (see the proof of Lemma \ref{Slsc}). Via the Lebesgue dominated convergence theorem, from $u_n\in U_{w_n}$ it follows $u\in U_w$. Plugging all together, we get $u \in \mathscr{S}(w)$. It remains to verify that $u =\Gamma(w)$. Lemma \ref{Slsc} provides a sequence
$\{v_n\}\subseteq C^1(\overline{\Omega})$ fulfilling both $v_n\in\mathscr{S}(w_n)$ for all $n\in\N$ and $v_n\to\Gamma(w)$ in
$ C^1(\overline{\Omega})$. The choice of $\Gamma$ entails $u_n=\Gamma(w_n) \leq v_n$, besides $\Gamma(w)\leq u$. Letting $n \to+\infty$ we thus arrive at
\[
\Gamma(w)\leq u=\lim_{n\to+\infty} u_n\leq\lim_{n \to+\infty} v_n=\Gamma(w),
\]
i.e., $u=\Gamma(w)$, which completes the proof.
\end{proof}
To establish our main result, the stronger version below of  H(f) will be employed.
\vskip3pt
\noindent $\underline{{\rm H'(f)}}$ $ f:\Omega \times\R\times\R^N\to [0,+\infty) $ is a Carathéodory function such that
\[
f(x,s,\xi)\leq c_3 + c_4 |s|^{p-1} + c_5 |\xi|^{p-1} \quad\forall\, (x,s,\xi) \in \Omega \times \R \times \R^N,
\]
with appropriate $c_3,c_4,c_5>0$.
\vskip3pt
\noindent Condition \eqref{condition} is substituted by
\begin{equation}\label{condition2} 
c_4 +(2p-1)c_5+ d< c_1^p c_2\, .
\end{equation}
\begin{rmk}
Assumption ${\rm H'(f)}$ clearly implies $ {\rm H(f)}$, with $ c_M:=c_3+c_5 M^{p-1}$ and $d_M:=c_4$. Likewise, \eqref{condition2} forces \eqref{condition} while \eqref{reccond} reads as
\begin{equation}
\label{reccond2}
c_4+d<\frac{c_1^p c_2}{p}\, .
\end{equation}
\end{rmk}
\begin{thm}
\label{existence}
Let ${\rm H'(f)}$, ${\rm H(g)}$, and \eqref{condition2}--\eqref{reccond2} be satisfied.
Then problem \eqref{problem} possesses a solution $u\in{\rm int}(C^1(\overline{\Omega})_+) $. The set of solutions to \eqref{problem} is compact in $C^1(\overline{\Omega})$.
\end{thm}
\begin{proof}
Define
\[
\Lambda(\Gamma):=\{ u \in C^1(\overline{\Omega}): u=\tau\,\Gamma(u)\;\mbox{for some}\;\tau\in (0,1)\}.
\]
\underline{Claim}: $\Lambda(\Gamma)$ is bounded in $W^{1,p}(\Omega)$.\\
To see this, pick any $u\in\Lambda(\Gamma)$. Since $\frac{u}{\tau}=\Gamma(u)\in\mathscr{S}(u)$, one has $\mathscr{E}_u \left( \frac{u}{\tau} \right)<1$. Assumption ${\rm H'(f)}$, combined with Young's and H\"older's inequalities, produces
\begin{equation*}
\begin{split}
\int_{\Omega \left(\frac{u}{\tau}>\underline{u}\right)}\left(\int_{\underline{u}}^{\frac{u}{\tau}} f(\cdot,t,\nabla u) {\rm d}t \right) {\rm d}x 
&\leq\int_{\Omega}\left( \int_{0}^{\frac{u}{\tau}} (c_3 + c_4 t^{p-1}+ c_5|\nabla u|^{p-1}) {\rm d}t \right) {\rm d}x \\
&\leq c_3 \left\|\frac{u}{\tau}\right\|_1 + \frac{c_4}{p} \left\|\frac{u}{\tau}\right\|_p^p+ c_5 \int_{\Omega} |\nabla u|^{p-1} \left|\frac{u}{\tau}\right| {\rm d}x \\
&\leq c_3 |\Omega|^{\frac{1}{p'}}\left\|\frac{u}{\tau}\right\|_p+\frac{c_4}{p} \left\|\frac{u}{\tau}\right\|_p^p+c_5 \left(\frac{\left\|\frac{u}{\tau}\right\|_p^p}{p}+\frac{\|\nabla u\|_p^p}{p'} \right) \\
&\leq c_3|\Omega|^{\frac{1}{p'}}\left\|\frac{u}{\tau}\right\|_{1,p}+\frac{c_4+c_5}{p}\left\|\frac{u}{\tau}\right\|_{1,p}^p + \frac{c_5}{p'}\|u\|_{1,p}^p. 
\end{split}
\end{equation*}
Analogously, on account of \eqref{subsupnorm},
\begin{equation*}
\begin{split}
\int_\Omega f(\cdot,\underline{u},\nabla u)\underline{u} {\rm d}x
&\leq\int_\Omega\left( c_3 \underline{u}+c_4 \underline{u}^p+c_5 |\nabla u|^{p-1}\right)\underline{u} {\rm d}x \\
&\leq \left( c_3 + c_4 + \frac{c_5}{p} \right) |\Omega| + \frac{c_5}{p'} \|\nabla u\|_p^p \\
&\leq \left( c_3 + c_4 + \frac{c_5}{p} \right) |\Omega| + \frac{c_5}{p'} \|u\|_{1,p}^p. \\
\end{split}
\end{equation*}
Reasoning as in Lemma \ref{estlemma} and recalling that $\tau\in (0,1)$, we thus achieve
\begin{equation*}
\begin{split}
1&>\mathscr{E}_u\left( \frac{u}{\tau} \right)\\
&\geq\frac{c_1^p c_2 - c_4 - (2p-1)c_5-d}{p}\left\|\frac{u}{\tau}\right\|_{1,p}^p
-(c_3 + c) |\Omega|^{\frac{1}{p'}}\left\| \frac{u}{\tau} \right\|_{1,p} - K',
\end{split}
\end{equation*}
where
$$K':=\left( c_3+c_4+\frac{c_5}{p} \right)|\Omega|+2 \|g(\cdot,\epsilon \theta)\|_{p'} |\Omega|^{\frac{1}{p}}.$$
Thanks to \eqref{condition2}, the above inequalities force 
\[
\| u\|_{1,p}\leq\left\| \frac{u}{\tau} \right\|_{1,p} \leq K^*,
\]
with $K^*>0$ independent of $u$ and $\tau$. Thus, the claim is proved.\\

By regularity \cite{Li}, the set $\Lambda(\Gamma)$ turns out bounded in $C^1(\overline{\Omega})$. Hence, due to Lemma \ref{propgamma}, Theorem \ref{schaefer} applies, which entails ${\rm Fix}(\Gamma) \neq \emptyset$. Let $u\in {\rm Fix}(\Gamma) $. From $u =\Gamma(u)\in\mathscr{S}(u)$ we deduce both $u\geq\underline{u}$ and $u\in U_u$. Accordingly, 
$$\hat{f}(\cdot,u)=f(\cdot,u,\nabla u),\quad\hat{g}(\cdot,u) = g(\cdot,u),$$
namely the function $u$ solves problem \eqref{problem}. Further, $u\in{\rm int}(C^1(\overline{\Omega})_+)$ because of the strong maximum principle.

Finally, arguing as in Lemma \ref{regularity} ensures that each solution to \eqref{problem} lies in $C^{1,\alpha}(\overline{\Omega})$. Since $C^{1,\alpha}(\overline{\Omega})\hookrightarrow C^1(\overline{\Omega})$ compactly and the solution set of \eqref{problem} is closed in $C^1(\overline{\Omega})$, the conclusion follows.
\end{proof}
\begin{rmk}
The same techniques can be applied for finding solutions to the Neumann problem
\begin{equation*}
\left\{
\begin{array}{ll}
- {\rm \, div} a(\nabla u) + |u|^{p-2}u = f(x,u,\nabla u) + g(x,u)\;\;&\mbox{in}\;\; \Omega, \\
u > 0\;\;&\mbox{in}\;\;\Omega, \\
\displaystyle{\frac{\partial u}{\partial \nu_a}} = 0\;\;&\mbox{on}\;\;\partial \Omega.
\end{array}
\right.
\end{equation*}
In fact, it is enough to replace the norm $\|\cdot\|_{\beta,1,p}$ with the standard one $\|\cdot\|_{1,p}$.
\end{rmk}
\section{Uniqueness (for $p=2$)}\label{S4}
Throughout this section, $p=2$, the operator $a$ fulfills H(a), while the nonlinearities $f$ and $g$ comply with H(f) and H(g), respectively. The following further conditions will be posited:
\begin{itemize}
\item[$({\rm a}_4)$] There exists $c_6\in (0,1]$ such that
\[
\langle a(\xi)-a(\eta),\xi-\eta\rangle\geq c_6 |\xi - \eta|^2\quad\forall\,\xi,\eta \in \R^N.
\]
\item[$\underline{{\rm H''(f)}}$] With appropriate $c_7,c_8>0$ one has
\begin{equation}\label{effeone}
[f(x,s,\xi)-f(x,t,\xi)](s-t)\leq c_7|s-t|^2
\end{equation}
\begin{equation}\label{effetwo}
|f(x,t,\xi) - f(x,t,\eta)| \leq  c_8 |\xi-\eta|
\end{equation}
in $\Omega\times\R\times\R^N$.
\item[$\underline{{\rm H'(g)}}$] There is $c_9>0$ such that
\begin{equation}\label{gone}
[g(x,s) - g(x,t)](s-t)\leq c_9 |s-t|^2\;\;\forall\, x \in\Omega,\; s,t \in [1,+\infty).
\end{equation}
Moreover,
\begin{equation}\label{gtwo}
g(x,s)\leq g(x,1)\;\;\mbox{in}\;\;\Omega\times (1,+\infty).
\end{equation}
\end{itemize}
\begin{ex}
The parametric $(2,q)$-Laplacian $\Delta+\mu\Delta_q$, where $1<q< 2$, $\mu\geq 0$, satisfies ${\rm H(a)}$ and $({\rm a}_4)$; cf. \cite[Lemma A.0.5]{P}.
\end{ex}
\begin{thm}
Under the above assumptions, problem \eqref{problem} admits a unique solution provided
\begin{equation}\label{uniqcond}
c_7 + c_1 c_8 + c_9 < c_1^2 c_6.
\end{equation}
\end{thm}
\begin{proof}
Suppose $u,v$ solve \eqref{problem}, test with $u-v$, and subtract to arrive at
\begin{equation}\label{weakcomp}
\begin{split}
&\int_{\Omega}\langle a(\nabla u)-a(\nabla v),\nabla(u-v)\rangle{\rm d}x+\beta\int_{\partial \Omega} |u-v|^2 {\rm d}\sigma\\
&= \int_{\Omega} [f(\cdot,u,\nabla u) - f(\cdot,v,\nabla v)](u-v){\rm d}x\\
&+ \int_{\Omega} [g(\cdot,u) - g(\cdot,v)](u-v) {\rm d}x.
\end{split}
\end{equation}
The left-hand side of \eqref{weakcomp} can easily be estimated from below via $({\rm a}_4)$ as follows:
\begin{equation}\label{aestimate}
\int_{\Omega}\langle a(\nabla u)-a(\nabla v),\nabla(u-v)\rangle {\rm d}x+\beta\int_{\partial \Omega}|u-v|^2 {\rm d}\sigma
\geq c_6 \|u-v\|_{\beta,1,2}^2.
\end{equation}
Using \eqref{effeone}--\eqref{effetwo} and H\"older's inequality we get
\begin{equation}\label{festimate}
\begin{split}
\int_{\Omega} & [f(\cdot,u,\nabla u) - f(\cdot,v,\nabla v)](u-v) {\rm d}x \\
&=\int_{\Omega} [f(\cdot,u,\nabla u) - f(\cdot,v,\nabla u)](u-v) {\rm d}x \\
&\phantom{pppppp}+\int_{\Omega}[f(\cdot,v,\nabla u) - f(\cdot,v,\nabla v)](u-v) {\rm d}x \\
&\leq c_7\int_{\Omega} |u-v|^2 {\rm d}x + c_8\int_{\Omega} |\nabla u - \nabla v| |u-v| {\rm d}x \\
&\leq c_7 \|u-v\|_2^2 + c_8 \|\nabla(u-v)\|_2\|u-v\|_2 \\
&\leq \frac{c_7}{c_1^2} \|u-v\|_{\beta,1,2}^2+ \frac{c_8}{c_1} \|u-v\|_{\beta,1,2}^2.
\end{split}
\end{equation}
Observe now that
\begin{equation}\label{gestimate}
\begin{split}
&\int_{\Omega} [g(\cdot,u) - g(\cdot,v)](u-v) {\rm d}x \\
&= \int_{\Omega(\max\{u,v\} \leq 1)} [g(\cdot,u) - g(\cdot,v)](u-v) {\rm d}x \\
&+ \int_{\Omega(\min\{u,v\} > 1)} [g(\cdot,u) - g(\cdot,v)](u-v) {\rm d}x \\
&+ \int_{\Omega(u \leq 1 < v)} [g(\cdot,u) - g(\cdot,v)](u-v) {\rm d}x \\
&+ \int_{\Omega(v \leq 1 < u)} [g(\cdot,u) - g(\cdot,v)](u-v) {\rm d}x.
\end{split}
\end{equation}
By hypothesis $({\rm g}_1)$ in H(g) one has
\begin{equation}\label{integral1}
\int_{\Omega(\max\{u,v\} \leq 1)} [g(\cdot,u) - g(\cdot,v)](u-v) {\rm d}x \leq 0.
\end{equation}
Inequality \eqref{gone} entails
\begin{equation}\label{integral2}
\begin{split}
&\int_{\Omega(\min\{u,v\} > 1)} [g(\cdot,u) - g(\cdot,v)](u-v) {\rm d}x \\
&\leq c_9 \|u-v\|_2^2 \leq \frac{c_9}{c_1^2} \|u-v\|_{\beta,1,2}^2.
\end{split}
\end{equation}
Thanks to $({\rm g}_1)$ again and \eqref{gtwo} we obtain
\begin{equation}\label{integral3}
\begin{split}
&\int_{\Omega(u \leq 1 < v)} [g(\cdot,u) - g(\cdot,v)](u-v) {\rm d}x \\
&\leq \int_{\Omega(u \leq 1 < v)} [g(\cdot,1) - g(\cdot,v)](u-v) {\rm d}x \leq 0.
\end{split}
\end{equation}
Likewise,
\begin{equation}\label{integral4}
\int_{\Omega(v \leq 1<u)} [g(\cdot,u) - g(\cdot,v)](u-v) {\rm d}x\leq 0.
\end{equation}
Plugging \eqref{integral1}--\eqref{integral4} into \eqref{gestimate} and \eqref{aestimate}--\eqref{gestimate} into \eqref{weakcomp}
yields
\begin{equation*}
c_6 \|u-v\|_{\beta,1,2}^2 \leq\left(\frac{c_7}{c_1^2} + \frac{c_8}{c_1} + \frac{c_9}{c_1^2}\right) \|u-v\|_{\beta,1,2}^2.
\end{equation*}
On account of \eqref{uniqcond}, this directly leads to $u=v$, as desired.
\end{proof}
\begin{rmk}
The conditions that guarantee existence or uniqueness, namely \eqref{condition2}, \eqref{reccond2}, and \eqref{uniqcond}, represent a balance between data (growth or variation of reaction terms) and structure (driving operator and domain) of the problem .
\end{rmk}
\section*{Acknowledgement}
This work is performed within PTR 2018--2020 - Linea di intervento 2: `Metodi Variazionali ed Equazioni Differenziali' of the University of Catania and partly funded by Research project of MIUR (Italian Ministry of Education, University and Research) Prin 2017 `Nonlinear Differential Problems via Variational, Topological and Set-valued Methods' (Grant Number 2017AYM8XW)


\begin{thebibliography}{777}
%
\bibitem{CLM} S. Carl, V. K. Le, and D. Motreanu, \textit{Nonsmooth Variational Problems and Their Inequalities,} Springer Monogr. Math., Springer, New York, 2007.
%
\bibitem{CRT} M.G. Crandall, P.H. Rabinowitz, and L. Tartar, \textit{On a Dirichlet problem with a singular nonlinearity}, Comm. Partial Differential Equations \textbf{2} (1977), 193--222.
%
\bibitem{E} L.C. Evans, \textit{Partial Differential Equations}, Grad. Stud. Math. \textbf{19}, Amer. Math. Soc., Providence, RI, 1998.
%
\bibitem{FarMotPug} F. Faraci, D. Motreanu, and D. Puglisi, \textit{Positive solutions of quasi-linear elliptic equations with dependence on the gradient}, Calc. Var. \textbf{54} (2015), 525--538.
%
\bibitem{FMP} G. Fragnelli, D. Mugnai, and N.S. Papageorgiou, \textit{Positive and nodal solutions for parametric nonlinear Robin problems with indefinite potential}, Discrete Cont. Dyn. Sist. \textbf{36} (2016), 6133--6166.
%
\bibitem{GMP} U. Guarnotta, S.A. Marano, and N.S. Papageorgiou, \textit{Multiple nodal solutions to a Robin problem with sign-changing potential and locally defined reaction}, Rend. Lincei Mat. Appl. \textbf{30} (2019), 269--294.
%
\bibitem{GP} L. Gasi\'nski and N.S. Papageorgiou, \textit{Nonlinear Analysis}, Chapman \& Hall/CRC, Boca Raton, 2006.
%
\bibitem{GP2} L. Gasi\'nski and N.S. Papageorgiou, \textit{Nonlinear elliptic equations with singular terms and combined nonlinearities}, Ann. Inst. H. Poincaré Anal. Non Lin\'{eaire} \textbf{13} (2012), 481--512.
%
\bibitem{KJF} A. Kufner, O. John, S. Fu\v{c}\'{i}k, \textit{Function spaces}, Publishing House of the Czechoslovak Academy of Sciences, Prague, 1977.
%
\bibitem{Le} A. Le, \textit{Eigenvalue problems for the $p$-Laplacian}, Nonlinear Anal. \textbf{64} (2006), 1057--1099.
%
\bibitem{Li} G. Liebermann, \textit{The natural generalization of the natural conditions of Ladyzhenskaya and Ural'tseva for elliptic equations}, Comm. Partial Differential Equations \textbf{16} (1991), 311--361.
%
\bibitem{LM} A.C. Lazer and P.J. McKenna, \textit{On a singular nonlinear elliptic boundary-value problem}, Proc. Amer. Math. Soc. \textbf{111} (1991), 721--730.
%
\bibitem{LMZ} Z. Liu, D. Motreanu, and S. Zeng, \textit{Positive solutions for nonlinear singular elliptic equations of $p$-Laplacian type with dependence on the gradient}, Calc. Var. \textbf{58} (2019), Paper No. 28, 22 pp.
%
\bibitem{MMM} D. Motreanu, V.V. Motreanu, and A. Moussaoui, \textit{Location of nodal solutions for quasilinear elliptic equations with gradient dependence}, Discrete Cont. Dyn. Sist. Series S \textbf{11} (2018), 293--307.
%
\bibitem{MW} D. Motreanu and P. Winkert, \textit{Existence and asymptotic properties for quasilinear elliptic equations with gradient dependence}, Appl. Math. Lett. \textbf{95} (2019), 78--84.
%
\bibitem{PapWink} N.S. Papageorgiou and P. Winkert, \textit{Singular $p$-Laplacian equations with superlinear perturbation,} J. Differential Equations \textbf{266} (2019), 1462--1487.
%
\bibitem{PW} N.S. Papageorgiou and P. Winkert, \textit{Solutions with sign information for nonlinear nonhomogeneous problems}, Math. Z. \textbf{292} (2019), 871--891.
%
\bibitem{P} I. Peral, \textit{Multiplicity of Solutions for the $p$-Laplacian,} ICTP Lecture Notes of the Second School of Nonlinear Functional Analysis and Applications to Differential Equations, Trieste, 1997.
%
\bibitem{PS} P. Pucci and J. Serrin, \textit{The maximum principle}, Birkh\"auser, Basel, 2007.
%
\bibitem{ZLM} S. Zeng, Z. Liu, S. Mig\'orski, \textit{Positive solutions to nonlinear nonhomogeneous inclusion problems with dependence on the gradient}, J. Math. Anal. Appl. \textbf{463} (2018), 432--448.
\end{thebibliography}
\end{document}